\documentclass[12pt]{amsart}

\usepackage{amsfonts,amssymb,amscd,amsmath,enumerate,verbatim,calc, xcolor}
\usepackage[pagebackref,colorlinks]{hyperref}
\usepackage[colorlinks]{hyperref}
\usepackage{amsrefs}

\usepackage{etoolbox}
\robustify{\fbox}%
\hypersetup{colorlinks=false}

\newtheorem{thm}{Theorem}[section]
\newtheorem{lem}[thm]{Lemma}
 
\newtheorem{prop}[thm]{Proposition}

\theoremstyle{definition}
\newtheorem{defi}[thm]{{Definition}}
\newtheorem{Example}[thm]{Example}

\newtheorem{chu}[thm]{}

\newtheorem*{chunk*}{}

\numberwithin{equation}{thm}
\theoremstyle{remark}
\newtheorem{rem}[thm]{Remark}

\makeatletter
\renewcommand{\boxed}[1]{\text{\fboxsep=.02em\fbox{\m@th$\displaystyle#1$}}}
\makeatother

\begin{document}

\title[Root closure of ideals]{Root closure of ideals}

\author[J.\, Forsman]{Joey Forsman}
\address{Department of Mathematics 2750\\ North Dakota State University\\PO BOX 6050\\ Fargo, ND 58108-6050\\ USA}
\email{john.forsman@ndsu.edu}

\date{20 October 2022}


\begin{abstract}  We introduce two closure operations on ideals in commutative rings related to the ring operation of root closure. One closure is the result of iterating a root-like operation on ideals infinitely many times, and the other closure arises as a graded component of the root closure of the Rees algebra applied to an ideal. We study the properties of these closures and prove they are distinct operations. 
\end{abstract}

\maketitle

\section{Introduction}

Let $A \subseteq B$ be an extension of commutative rings. Then $A$ is said to be \textit{root closed} in $B$ if whenever $x^n \in A$ for some $x \in B$ and positive integer $n$, then $x \in A$. The smallest subring of $B$ which contains $A$ and is root closed in $B$ is called the \textit{root closure} of $A$ in $B$.

The concept of root closedness was introduced by Sheldon \cite{Sh} in the context where $A$ is an integral domain and $B$ is its quotient field. Later, Anderson \cite{An} studied root closedness in the more general context where $A$ and $B$ are integral domains. In 1990, Anderson, Dobbs, and Roitman \cite{AnDoRo} gave the definition of root closure for a general extension of commutative rings, although they referred to it as total root closure. 

In this paper, for an arbitrary commutative ring $A$ and ideal $I \subseteq A$, we define several ideals $I^{\#}$, $I^{\boxed{\#}}$, and $I^\natural$ of $A$ which are analogous to the ring operation of root closure. Throughout this paper, all rings are commutative and have an identity element.

\begin{defi}[\cite{AnDoRo}]\label{def-ring} Let $A \subseteq B$ be an extension of rings. Define 
\begin{itemize}
\item $A^{\#} = A[x \in B \mid x^n \in A \text{ for some $n \geq 1$} ]$.
\item $A^{\#_1} = A^{\#}$. For an integer $\ell > 1$, define $A^{\#_\ell} = A^{\#_{\ell - 1}}[x \in B \mid x^n \in A^{\#_{\ell - 1}} \text{ for some $n \geq 1$} ]$.
\item $A^\boxed{\#} = \bigcup\limits_{\ell \geq 1} A^{\#_{\ell}}$.
\end{itemize}
\end{defi}

\begin{rem}
The subring $A^{\#}$ of $B$ admits an alternative description: $$A^{\#} = \{ \text{finite sums } \sum\limits_{i} b_i \mid b_i \in B \text{ and } b_i^n \in A \text{ for some }   n \geq 1 \}.$$ 
To see this, observe that every element of $A^{\#}$ is a polynomial in elements of $B$ for which some power of that element is in $A$. Each term of this polynomial also has some power belonging to $A$, as each term is a product of such elements. We take the terms of this polynomial to be the finitely many $b_i$ in the displayed description.
\end{rem}

Anderson, Dobbs, and Roitman \cite[Proposition 1.1]{AnDoRo} prove the ascending union $A^{\boxed{\#}}$ is the root closure of $A$ in $B$. In that paper, they denote this ascending union as $\mathfrak{R}^{\mathbb{N}}_\infty(A,B)$.

\begin{prop}[\cite{AnDoRo}] Let $A \subseteq B$ be an extension of rings. Then $A^{\boxed{\#}}$ is the root closure of $A$ in $B$.
\end{prop}


The definitions of the following ideals are inspired from Definition \ref{def-ring}.

\begin{defi}\label{def-ideal} Let $A$ be a ring and let $I \subseteq R$ be an ideal. Define 
\begin{itemize}
\item $I^{\#} = \langle x \in A \mid x^n \in I^n \text{ for some $n \geq 1$} \rangle$.
\item $I^{\#_1} = I^{\#}$. For an integer $\ell > 1$, define $I^{\#_\ell} = \langle x \in A \mid x^n \in (I^{\#_{\ell - 1}})^n \text{ for some $n \geq 1$} \rangle$.
\item $I^\boxed{\#} = \bigcup\limits_{\ell \geq 1} I^{\#_{\ell}}$. We refer to $I^{\boxed{\#}}$ as the \textit{boxed-sharp closure} of $I$.
\end{itemize}
\end{defi}

\begin{rem}
The ideal $I^{\#}$ is \textit{generated} by the set of those elements such that some power of the element belongs to the same power of $I$. This set does not, in general, form an ideal. Consider the ring $A = \mathbb{C}[X, Y]$, the ideal $I = \langle X^2, Y^2 \rangle$,  and the set $$\mathbb{S} = \{ x \in A \mid x^n \in I^n \text{ for some } n \geq 1 \}.$$ Then $X^2, 2XY, Y^2 \in \mathbb{S}$, yet $(X + Y)^2 \not\in \mathbb{S}$. To see this, suppose otherwise that there exists $n \geq 1$ such that $(X + Y)^{2n} \in I^n$. By the binomial theorem, the coefficient of $XY^{2n-1}$ in the expansion of $(X + Y)^{2n}$ is $2n$. However, by taking the degree $2n$ component across the inclusion $(X + Y)^{2n} \in I^n = \langle X^2, Y^2 \rangle^n$, we find by comparing coefficients on the term $XY^{2n - 1}$ that $2n = 0$, a contradiction.
\end{rem}

We adapt the notion of root closure to the case of ideals with the following definition.

\begin{defi}\label{def-ideal-root} Let $A$ be a ring and let $I \subseteq A$ be an ideal. We say $I$ is \textit{root closed} in $A$ if whenever $x^n \in I^n$ for some $x \in A$ and positive integer $n$, then $x \in I$. The smallest ideal of $A$ which contains $I$ and is root closed in $A$ is called the \textit{root closure} of $I$ in $A$.
\end{defi}

Just as $A^{\boxed{\#}}$ is the root closure of $A$ in $B$, we prove in Section \ref{sec-closure} for a given ideal $I \subseteq A$ that $I^{\boxed{\#}}$ is the root closure of $I$ in $A$. Here we have a notion ``root closure'' characterized for rings and for ideals. The operation ``integral closure'' also has a meaning for rings and for ideals, and this can be exhibited through the Rees algebra of the ideal $I$. Namely, for an ideal $I \subseteq A$ and indeterminate $t$ over $A$, we study the integral closure of $R[It]$ in $R[t]$. This ring, denoted $\overline{R[It]}$, is $\mathbb{N}$-graded and has graded structure $$\overline{R[It]} = \bigoplus\limits_{n \in \mathbb{N}} \overline{I^n}t^n,$$ where $\overline{I^n}$ denotes the integral closure of the ideal $I^n$. Mimicking this perspective, in Section \ref{sec-rees}, we determine the graded structure of the root closure of the Rees algebra $R[It]$ in $R[t]$. The root closure $R[It]^{\boxed{\#}}$ of $R[It]$ in $R[t]$ is $\mathbb{N}$-graded and has graded structure $$R[It]^{\boxed{\#}} = \bigoplus\limits_{n \in \mathbb{N}} (I^n)^{\natural}t^n,$$ where, for a given ideal $I \subseteq A$, the ideal $I^\natural$ is defined as follows.

\begin{defi} Let $A$ be a ring and let $I \subseteq A$ be an ideal. The \textit{natural closure} $I^\natural$ of $I$ in $A$ is defined by $$I^\natural = \{ x \in A \mid x^n \in (I^n)^{\boxed{\#}} \text{ for some } n \geq  1 \}.$$
\end{defi}

In Section \ref{sec-closure}, we prove this is a closure operation on ideals. We also prove various properties the operations $\#$, $\boxed{\#}$, and $\natural$ enjoy. In Section \ref{sec-ex}, we provide examples which show that $\#$, ${\boxed{\#}}$, $\natural$, and integral closure are distinct operations on ideals.

\section{Properties of the operations $\#$, \protect{\boxed{\#}}, and $\natural$ }\label{sec-closure}

We recall the notion of a closure operation on ideals. 
\begin{defi}\label{defi-closure-op}
Let $A$ be a ring. A \textit{closure operation} $\square$ on the set of ideals of $A$ is a mapping $I \xrightarrow{} I^{\square}$ which satisfies the following conditions:
\begin{enumerate}
\item For all ideals $I$ of $A$, $I \subseteq I^{\square}.$
\item For all ideals $I, J$ of $A$ with $I \subseteq J$, $I^\square \subseteq J^{\square}.$
\item For all ideals $I$ of $A$, $I^{\square} = (I^{\square})^{\square}$.
\end{enumerate}
\end{defi}

\begin{prop}\label{prop-sharp-closure} Let $A$ be a ring. Then the mapping on ideals of $A$ given by $I \xrightarrow{} I^{\#}$ satisfies conditions (1) and (2) of Definition \ref{defi-closure-op}, but does not generally satisfy condition (3).
\end{prop}
\noindent It is clear from the definition that the mapping $I \xrightarrow{} I^{\#}$ satisfies conditions (1) and (2). In Example \ref{exb}, we exhibit an ideal $I$ such that $I^{\#} \neq (I^{\#})^{\#}$.

In Propositions \ref{prop-box-closure} and \ref{prop-nat-closure} respectively, we prove that $\boxed{\#}$ and $\natural$ are closure operations on ideals. For a given ideal $I$ of a ring $A$, we need to prove  the sets $ I^{\boxed{\#}}$ and $I^\natural$ are ideals. For all $\ell >1$, Proposition \ref{prop-sharp-closure} proves $I^{\# _\ell} \subseteq (I^{\#_\ell})^{\#} =  I^{\#_{\ell + 1}}$. Thus $I^{\boxed{\#}}$ is a union over an increasing chain of ideals and is thus an ideal itself.  Showing $I^\natural$ is an ideal is more involved; we first prove two lemmas.

\begin{lem}\label{lem-contains}
Let $A$ be a ring and let $I, J$ be ideals of $A$. Then the following containments hold: $$I^{\#}J^{{\#}} \subseteq (IJ)^{{\#}} \text{ and } I^{\boxed{\#}}J^{\boxed{\#}} \subseteq (IJ)^{\boxed{\#}}.$$
\end{lem}
\begin{proof}
For the first part, let $xy \in I^{\#}J^{{\#}}$, where $x \in I^{\#}$ and $y \in J^{{\#}}$. Assume further that $x^{n_1} \in I^{n_1}$ for some ${n_1} \geq 1$ and $y^{n_2} \in J^{n_2}$ for some $n_2 \geq 1$. Since elements of the form $xy$ generate $I^{\#}J^{\#}$, it is enough to show $xy \in (IJ)^{{\#}}$. Note that $(xy)^{{n_1}{n_2}} = (x^{n_1})^{n_2}(y^{n_2})^{n_1} \in (I^{n_1})^{n_2}(J^{n_2})^{n_1} = I^{{n_1}{n_2}}J^{{n_1}{n_2}} =  (IJ)^{{n_1}{n_2}}$. Hence $xy \in (IJ)^{\#}$ as desired.

The first part shows that for $\ell > 1$ we have $I^{\#_\ell}J^{\#_\ell} = (I^{\#_{\ell - 1}})^{\#}(J^{\#_{\ell - 1}})^{\#} \subseteq (I^{\#_{\ell-1}}J^{\#_{\ell -1}})^{\#}$. Applying the first part, we find $(I^{\#_{\ell-1}}J^{\#_{\ell -1}})^{\#} \subseteq (I^{\#_{\ell-2}}J^{\#_{\ell -2}})^{\#_2}$. Continuing to apply the first part iteratively, we find $I^{\#_\ell}J^{\#_\ell} \subseteq (IJ)^{\#_\ell}$.

For the second part, let $xy \in I^{\boxed{\#}}J^{\boxed{\#}}$,  where $x \in I^{\boxed{\#}}$ and $y \in J^{\boxed{\#}}$. It is enough to show $xy \in (IJ)^{\boxed{\#}}$. We may write $x \in I^{\#_{\ell_1}}$ for some $\ell_1 \geq 1$ and $y \in J^{\#_{\ell_2}}$ for some $\ell_2 \geq 1$. Write $\ell = \max \{ \ell_1, \ell_2 \}$. Now $I^{\#_{\ell_1}} \subseteq I^{\#_{\ell}}$ and  $J^{\#_{\ell_2}} \subseteq J^{\#_{\ell}}$, so $xy \in I^{\#_{\ell}} J^{\#_{\ell}} \subseteq  (IJ)^{\#_{\ell}} \subseteq (IJ)^{\boxed{\#}}$.
\end{proof}

\begin{lem}\label{lem-box-sharp-sharp}
Let $A$ be a ring and let $I \subseteq A$ be an ideal. Then $(I^{\boxed{\#}})^{\#} = I^{\boxed{\#}}$.
\end{lem}
\begin{proof}
The reverse containment is clear by Proposition \ref{prop-sharp-closure}.  For the forward containment, let $x \in (I^{\boxed{\#}})^\#$. Write $x = x_1 + \dots + x_r$,  where $x_i^{n_i} \in (I^{\boxed{\#}})^{n_i}$ for some $n_i \geq 1$ for all $i$. Define $n \in \mathbb{N}$ to be the product of the elements $n_1, \dots, n_r$. Notice $x_i^n \in (I^{\boxed{\#}})^n$ for all $i$.  We may find $\ell \geq 1$ such that $x_i^n \in (I^{\#_\ell})^n$ for all $i$. Indeed, since $x_i^n \in (I^{\boxed{\#}})^n$, we may write $x_i^n$ as a finite sum of terms which are $n$-fold products of elements in $I^{\boxed{\#}}$. Each such element belongs to $I^{\#_\ell}$ for some large enough $\ell$, and thus each such $n$-fold product belongs to $(I^{\#_\ell})^n$.  Now, since $x_i^n \in (I^{\#_\ell})^n$ for all $i$, we find $x_i \in I^{\#_{\ell+1}} \subseteq I^{\boxed{\#}}$ for all $i$. It follows $x \in I^{\boxed{\#}}$.
\end{proof}

\begin{prop}
Let $A$ be a ring and let $I \subseteq A$ be an ideal. Then the set $I^\natural$ is an ideal of $A$.
\end{prop}
\begin{proof}
Let $x, y \in I^\natural$ and $r \in A$. We may write $x^{n_1} \in (I^{n_1})^{\boxed{\#}}$ and $y^{n_2} \in (I^{n_2})^{\boxed{\#}}$ for some $n_1, n_2 \geq 1$. Observe that $(rx)^{n_1} = r^{n_1}x^{n_1} \in (I^{n_1})^{\boxed{\#}}$, so $rx \in I^\natural$. Next, we show $x + y \in I^\natural$. To this end, let $n \in \mathbb{N}$ denote the product $n_1n_2$. Note $x^n = (x^{n_1})^{n_2} \in ((I^{n_1})^{\boxed{\#}})^{n_2} \subseteq (I^n)^{\boxed{\#}}$, where the last containment follows from Lemma \ref{lem-contains}. Similarly, $y^n \in (I^n)^{\boxed{\#}}$. Let $J$ denote the ideal $\langle x , y \rangle^n$, and consider the elements in $J$ of the form $x^{n-i}y^i$ where $0 \leq i \leq n$. The set of such elements generates $J$. For all $i$, $(x^{n-i}y^i)^n = (x^n)^{n-i}(y^n)^{i} \in  ((I^n)^{\boxed{\#}})^{n}$, which means $x^{n-i}y^i \in ((I^n)^{\boxed{\#}})^{\#} = (I^n)^{\boxed{\#}}$ for all $i$, where the last equality follows from Lemma \ref{lem-box-sharp-sharp}. Hence $J \subseteq (I^n)^{\boxed{\#}}$. Notice that $(x + y)^n \in J \subseteq (I^n)^{\boxed{\#}}$, so $x + y \in I^\natural$.
\end{proof}

Now we aim to prove ${\boxed{\#}}$ and $\natural$ are closure operations on ideals. Whether these are \textit{distinct} operations is answered positively in Section \ref{sec-ex}.

\begin{prop}\label{prop-box-closure} Let $A$ be a ring. Then the mapping on ideals of $A$ given by $I \xrightarrow{} I^{\boxed{\#}}$ is a closure operation.
\end{prop}
\begin{proof}

Let $I \subseteq A$ be an ideal. We have $I \subseteq I^{\#} \subseteq I^{\boxed{\#}}$. Next, let $J$ be an ideal such that $I \subseteq J$. To show $I^{\boxed{\#}} \subseteq J^{\boxed{\#}}$, let $x \in I^{\boxed{\#}}$. We may find $\ell \geq 1$ such that $x \in I^{\#_\ell}$. By applying the operation $\#$ across the containment $I \subseteq J$ a total of $\ell$ times, we obtain $I^{\#_\ell} \subseteq J^{\#_\ell}$. Thus $x \in J^{\#_\ell} \subseteq J^{\boxed{\#}}$, which shows $I^{\boxed{\#}} \subseteq J^{\boxed{\#}}$, as desired. Finally, we show $(I^{\boxed{\#}})^{\boxed{\#}} = I^{\boxed{\#}}$. By Lemma \ref{lem-box-sharp-sharp}, we see $$I^{\boxed{\#}} = (I^\boxed{\#})^{\#} = ((I^\boxed{\#})^{\#})^{\#} = \dots = (I^\boxed{\#})^{\#_\ell} = \dots, $$ where $\ell \geq 1$. Hence $(I^{\boxed{\#}})^{\boxed{\#}} =  \bigcup\limits_{\ell \geq 1} (I^\boxed{\#})^{\#_{\ell}} = I^{\boxed{\#}}.$ 

\end{proof}

\begin{prop}
Let $A$ be a ring, and let $I \subseteq A$. Then the ideal $I^{\boxed{\#}}$ is the root closure of $I$. 
\end{prop}
\begin{proof}
First we prove that $I^{\boxed{\#}}$ is root closed. Let $x \in A$ be such that $x^n \in (I^{\boxed{\#}})^n$ for some $n \geq 1$. Then $x \in (I^{\boxed{\#}})^{\#} = I^{\boxed{\#}}$. 

Next we prove that $I^{\boxed{\#}}$ is the smallest root closed ideal of $R$ containing $I$. To this end, let $J \subseteq R$ be a root closed ideal containing $I$. Since $J$ is root closed, we claim that $J^{\#_\ell} = J$ for all $\ell \geq 1$. To this end, we start by proving $J^{\#} = J$. Let $x \in A$ be such that $x^n \in J^n$ for some $n \geq 1$. Since $J^\#$ is generated by elements of this type, it suffices to show $x \in J$. This is clear, since $J$ is root closed. Since $J = J^{\#}$, for $\ell \geq 1$, $$J = J^{\#} = (J^{\#})^{\#} = \dots = J^{\#_\ell} = \dots.$$ Next, we prove $I^{\boxed{\#}} \subseteq J$.  Let $x \in I^{\boxed{\#}}$. We may write $x \in I^{\#_\ell}$ for some $\ell \geq 1$. Apply $\#$ to the containment $I \subseteq J$ to obtain $I^{\#} \subseteq J^{\#}= J$. Applying a total of $\ell$ times, we find $I^{\#_\ell} = J$. We see $x \in J$, which shows $I^{\boxed{\#}} \subseteq J$, as desired.
\end{proof}

We need two lemmas before we can prove $\natural$ is a closure operation.

\begin{lem}\label{lem-contains-nat}
Let $A$ be a ring and let $I, J$ be ideals of $A$. Then the following containment holds: $$I^\natural J^\natural \subseteq (IJ)^\natural.$$
\end{lem}
\begin{proof}
Let $xy \in I^\natural J^\natural$, where $x \in I^\natural$ and $y \in J^\natural$. The ideal $I^\natural J^\natural$ is generated by elements of the form $xy$, so it suffices to show $xy \in  (IJ)^\natural$. We may write  $x^{n_1} \in (I^{n_1})^{\boxed{\#}}$ and $y^{n_2} \in (J^{n_2})^{\boxed{\#}}$ for some $n_1, n_2 \geq 1$. Then $(xy)^{n_1n_2} = (x^{n_1})^{n_2}(y^{n_2})^{n_1} \in ((I^{n_1})^{\boxed{\#}})^{n_2}((J^{n_2})^{\boxed{\#}})^{n_1}$. By Lemma \ref{lem-contains}, $(xy)^{n_1n_2} \in  (I^{n_1n_2})^{\boxed{\#}}(J^{n_1n_2})^{\boxed{\#}} \subseteq (I^{n_1n_2}J^{n_1n_2})^{\boxed{\#}} =  ((IJ)^{n_1n_2})^{\boxed{\#}}$. Hence $xy \in (IJ)^\natural, $ as desired.
\end{proof}

\begin{lem}\label{lem-nat-sharp-sharp}
Let $A$ be a ring and let $I \subseteq A$ be an ideal. Then $(I^{\natural})^{\#} = I^\natural$ and $(I^{\natural})^{\boxed{\#}} = I^\natural$.
\end{lem}
\begin{proof}
We first show $(I^{\natural})^{\#} = I^\natural$. Let $x \in A$ be such that $x^{n_1} \in (I^\natural)^{n_1}$ for some $n_1 \geq 1$. Since $(I^\natural)^{\#}$ is generated by elements of this form, it suffices to show $x \in I^\natural$. By Lemma \ref{lem-contains-nat}, we see $x^{n_1} \in  (I^{n_1})^\natural$, so there exists some $n_2 \geq 1$ such that $(x^{n_1})^{n_2} \in ((I^{n_1})^{n_2})^{\boxed{\#}}$. Thus $x^{n_1n_2} \in (I^{n_1n_2})^{\boxed{\#}}$, so $x \in I^\natural$.

To show $(I^{\natural})^{\boxed{\#}} = I^\natural$, observe by the first part that $$I^\natural = (I^\natural)^{\#} = ((I^\natural)^{\#})^{\#} = \dots = (I^\natural)^{\#_\ell} = \dots,$$ where $\ell \geq 1$. Hence $(I^{\natural})^{\boxed{\#}} =  \bigcup\limits_{\ell \geq 1} (I^\natural)^{\#_{\ell}} = I^{\natural}.$
\end{proof}

Now we are ready to prove $\natural$ is a closure operation.

\begin{prop}\label{prop-nat-closure} Let $A$ be a ring. Then the mapping on ideals of $A$ given by $I \xrightarrow{} I^{\natural}$ is a closure operation.
\end{prop}
\begin{proof}

Let $I \subseteq A$ be an ideal. We have $I \subseteq I^{\boxed{\#}}$, so $I \subseteq I^\natural$. Next, let $J$ be an ideal such that $I \subseteq J$. To show $I^{\natural} \subseteq J^{\natural}$, let $x \in I^\natural$. We may write $x^n \in (I^n)^{\boxed{\#}}$ for some $n \geq 1$. Then $x^n \in (J^n)^{\boxed{\#}}$. Hence $x \in J^\natural$. Finally, we prove $(I^\natural)^\natural = I^\natural$. For the forward containment, let $x \in (I^\natural)^\natural$. Write $x^{n_1} \in ((I^\natural)^{n_1})^{\boxed{\#}}$ for some ${n_1} \geq 1$. By Lemma \ref{lem-contains-nat}, we have $(I^\natural)^{n_1} \subseteq (I^{n_1})^\natural$, from which it follows $((I^\natural)^{n_1})^{\boxed{\#}} \subseteq ((I^{n_1})^\natural)^{\boxed{\#}}$. By Lemma \ref{lem-nat-sharp-sharp}, we have  $((I^{n_1})^\natural)^{\boxed{\#}} = (I^{n_1})^\natural$. Thus $x^{n_1} \in (I^{n_1})^\natural$, so we may write $(x^{n_1})^{n_2} \in ((I^{n_1})^{n_2})^\boxed{\#}$ for some ${n_2} \geq 1$. Then $x^{{n_1}{n_2}} \in (I^{{n_1}{n_2}})^{\boxed{\#}}$, which shows $x \in I^\natural$.
\end{proof}

\begin{rem}\label{rem-containment}
For a ring $A$ and ideal $I \subseteq A$, we have introduced the ideals $I^{\#}$, $I^\boxed{\#}$, and $I^\natural$. Related to these is the integral closure $\overline{I}$ of $I$, defined as follows: $$\overline{I} = \{ x \in A \mid \text{ there exist } a_i \in I^i, n \in \mathbb{N}, \text{ such that } x^n + a_1x^{n-1} + \dots + a_{n-1}x + a_n = 0 \}.$$ Among these ideals, the following containment relations always hold: $$I \subseteq I^{\#} \subseteq I^{\boxed{\#}} \subseteq I^\natural \subseteq \overline{I}.$$ The first three containments are clear from the definitions. To see that $I^\natural \subseteq \overline{I}$, we first show that $I^{\#} \subseteq \overline{I}$. Let $x \in A$ be such that $x^n \in I^n$ for some $n \geq 1$; say $x^n = c$, where $c \in I^n$. To show $I^{\#} \subseteq \overline{I}$, it suffices to show $x \in \overline{I}$. But the equation $x^n - c = 0$ shows $x \in \overline{I}$. Next, let $\ell \geq 1$. We prove $I^{\#_\ell} \subseteq \overline{I}$. The base case $\ell = 1$ was just handled. Now assume $I^{\#_\ell} \subseteq \overline{I}$. For induction, we need to show $I^{\#_{\ell+1}} \subseteq \overline{I}$. Applying $\#$ to the containment $I^{\#_\ell} \subseteq \overline{I}$, we find $I^{\#_{\ell+1}} \subseteq (\overline{I})^{\#} \subseteq \overline{(\overline{I})} = \overline{I}$. Finally, let $x \in I^\natural$. We may write  $x^n \in (I^n)^\boxed{\#}$ for some $n \geq 1$. Then write $x^n \in (I^n)^{\#_\ell}$ for some $\ell \geq 1$. Now $(I^n)^{\#_\ell} \subseteq \overline{I^n}$, so we find $x^n \in \overline{I^n}$. By a well-known property of integral closure, we find $x \in \overline{I}$.

In Section \ref{sec-ex}, for each of the containments above, we find an example of a ring $A$ and ideal $I \subseteq A$ such that the containment is proper. These examples show that the closure operations $\boxed{\#}, \natural,$ and integral closure are distinct operations. However, in the case where $A$ is a polynomial ring over a field and $I \subseteq A$ is a monomial ideal, the ideals $I^{\#}$, $I^{\boxed{\#}}$, $I^\natural$, and $\overline{I}$ all coincide:
\end{rem}

\begin{rem}\label{rem-monomial}
Let $A$ be a polynomial ring over a field, and let $I \subseteq A$ be an ideal generated by monomials. Consider the well-known characterization of the integral closure $\overline{I}$  of $I$: $$\overline{I} = \langle m \in A \mid m \text{ is a monomial and } m^n \in I^n \text{ for some } n \geq 1 \rangle.$$ (A proof of this can be found in \cite[Proposition 12.1.2]{Vi}.) From this characterization, we find $\overline{I} \subseteq \langle x \in A \mid x^n \in I^n \text{ for some $n \geq 1$} \rangle = I^{\#}$.  As discussed in Remark \ref{rem-containment}, we have $I^{\#} \subseteq \overline{I}$. Hence we find that $I^{\#} = I^{\boxed{\#}} = I^\natural = \overline{I}$ in the monomial ideal case.
\end{rem}

We now prove a variety of useful properties of $\#$, $\boxed{\#}$, and $\natural$.

\begin{prop}[Persistence]\label{prop-persistence} Let $A \subseteq B$ be an extension of rings, and let $I \subseteq A$ be an ideal. The following containments hold:
\begin{enumerate}
\item $I^{\#}B \subseteq (IB)^{\#}$.
\item $I^{\boxed{\#}}B \subseteq (IB)^{\boxed{\#}}$.
\item $I^{\natural}B \subseteq (IB)^{\natural}$.
\end{enumerate}
\end{prop}
\begin{proof}
We begin by proving (1). The ideal $I^{\#}B$ is generated by elements of the form $xb$, where $x \in A$ satisfies $x^n$ for some $n \geq 1$ and $b \in B$. It suffices to show $xb \in (IB)^{\#}$. To this end, notice that $(xb)^n = x^nb^n \in I^nB = (IB)^n$. Thus $xb \in (IB)^{\#}$.

Before proving (2), we prove inductively that for all $\ell \geq 1$, $I^{\#_\ell}B \subseteq (IB)^{\#_\ell}$. The base case $\ell = 1$ was handled in (1). Now, fix $\ell$ and assume $I^{\#_\ell}B \subseteq (IB)^{\#_\ell}$. We need to show $I^{\#_{\ell+1}}B \subseteq (IB)^{\#_{\ell+1}}$. But $I^{\#_{\ell+1}}B = (I^{\#_{\ell}})^{\#}B \subseteq (I^{\#_{\ell}}B)^{\#} \subseteq ((IB)^{\#_\ell})^{\#} = (IB)^{\#_{\ell+1}}$, as desired.

Next, we prove (2). The ideal $I^{\boxed{\#}}B$ is generated by elements of the form $xb$, where $x \in I^{\boxed{\#}}$ and $b \in B$. It suffices to show $xb \in (IB)^{\boxed{\#}}$. We may write $x \in I^{\#_\ell}$ for some $\ell \geq 1$. Then $xb \in I^{\#_\ell}B \subseteq (IB)^{\#_\ell} \subseteq (IB)^{\boxed{\#}}$.

Finally, we prove (3). The ideal $I^{\natural}B$ is generated by elements of the form $xb$, where $x \in I^{\natural}$ and $b \in B$. It suffices to show $xb \in (IB)^\natural$. We may write $x^n \in (I^n)^{\boxed{\#}}$ for some $n \geq 1$. Then $(xb)^n = x^nb^n \in (I^n)^{\boxed{\#}}B \subseteq ((I^n)B)^{\boxed{\#}} = ((IB)^n)^{\boxed{\#}}$, so  $xb \in (IB)^\natural$.
\end{proof}

The operations $\#$, $\boxed{\#}$, and $\natural$ commute with localization.

\begin{prop}[Localization]\label{prop-localization} Let $A$ be a ring, let $S \subseteq A$ be a multiplicatively closed subset, and let $I \subseteq A$ be an ideal. Then the following equalities hold:
\begin{enumerate}
\item $S^{-1}I^{\#} = (S^{-1}I)^{\#}$.
\item $S^{-1}I^\boxed{\#} = (S^{-1}I)^{\boxed{\#}}$.
\item $S^{-1}I^\natural = (S^{-1}I)^{\natural}$.
\end{enumerate}
\end{prop}
\begin{proof}
The forward containments of (1), (2), and (3) each follow directly from Prop. \ref{prop-persistence}.

We begin by proving (1). Let $x \in  (S^{-1}I)^{\#}$. We may write $x = b_1 + \dots + b_r$ for some $b_i \in S^{-1}A$ satisfying $b_i^{n_i} \in S^{-1}I^{n_i}$ for all $i$. Define $n \in \mathbb{N}$ to be the product of the elements $n_1, \dots, n_r$. Notice $b_i^{n} \in S^{-1}I^{n}$ for all $i$. Clearing denominators, we see there exists $s \in S$ such that $sb_i \in A$ for all $i$ and $sb_i^n \in I^n$ for all $i$. Then $sx = sb_1 + \dots + sb_r$ is an equality which holds in $S^{-1}A$ where all terms are in $A$. So, there exists $s' \in S$ such that $s'sx = s'sb_1 + \dots + s'sb_r$ holds as an equality in $A$. Since $(s'sb_i)^n = (s')^ns^{n-1}(sb_i^n) \in I^n$ for all $i$, we find that $s'sx \in I^{\#}$. Hence $x \in S^{-1}I^{\#}.$

Before proving (2), we prove inductively that for all $\ell \geq 1$, $S^{-1}I^{\#_\ell} = (S^{-1}I)^{\#_\ell}$. The base case $\ell = 1$ was handled in (1). Now, fix $\ell$ and assume $S^{-1}I^{\#_\ell} = (S^{-1}I)^{\#_\ell}$. We need to show $S^{-1}I^{\#_{\ell+1}} = (S^{-1}I)^{\#_{\ell+1}}$.  But $S^{-1}I^{\#_{\ell+1}} = S^{-1}(I^{\#_{\ell}})^{\#} = (S^{-1}I^{\#_{\ell}})^{\#} = ((S^{-1}I)^{\#_\ell})^{\#} = (S^{-1}I)^{\#_{\ell+1}}$, as desired.

Next, we prove (2). Let $x \in (S^{-1}I)^{\boxed{\#}}$. We may write $x \in (S^{-1}I)^{\#_\ell}$ for some $\ell \geq 1$. Then $x \in S^{-1}I^{\#_\ell} \subseteq S^{-1}I^{\boxed{\#}}$, as desired.

Finally, we prove (3). Let $x \in (S^{-1}I)^\natural$. For some $n \geq 1$, we may write $x^n \in ((S^{-1}I)^n)^\boxed{\#} = (S^{-1}I^n)^{\boxed{\#}} = S^{-1}(I^n)^{\boxed{\#}}$, where the last equality follows from 2. By clearing denominators, we see there exists $s \in S$ such that $sx^n \in (I^n)^\boxed{\#}$, so $(sx)^n = s^nx^n \in (I^n)^\boxed{\#}$. Then $sx \in I^\natural$, so $x \in  S^{-1}I^\natural$.

\end{proof}

\begin{prop}[Colon]\label{prop-colon} Let $A$ be a ring and let $I, J$ be ideals of $A$. Then the following containments hold:
\begin{enumerate}
  \item $(I : J)^{\#} \subseteq (I^{\#} : J).$
  \item $(I : J)^{\boxed{\#}} \subseteq (I^{\boxed{\#}} : J).$
  \item $(I : J)^{\natural} \subseteq (I^{\natural} : J).$
\end{enumerate}
\end{prop}
\begin{proof}
We begin by proving (1). Let $x \in A$ be such that $x^n  \in (I : J)^{n}$. Since elements of this form generate $(I : J)^{\#}$, it suffices to show $x \in (I^{\#} : J)$. But $x^n \in (I : J)^n \subseteq (I^n : J^n)$, which implies $x^nJ^n \subseteq I^n$. Hence, for $j \in J$, we have $(xj)^n = x^nj^n \in I^n$ which implies $xj \in I^\#$; thus $x \in (I^\# : J)$, as desired.

Before proving (2), we prove inductively that for all $\ell \geq 1$, $(I : J)^{\#_\ell} \subseteq (I^{\#_\ell} : J).$ The base case $\ell = 1$ was handled in (1). Now, fix $\ell$ and assume $(I : J)^{\#_\ell} \subseteq (I^{\#_\ell} : J).$ We need to show $(I : J)^{\#_{\ell+1}} \subseteq (I^{\#_{\ell+1}} : J).$ But $(I : J)^{\#_{\ell+1}} = ((I : J)^{\#_{\ell}})^{\#} \subseteq (I^{\#_{\ell}} : J)^{\#} \subseteq (I^{\#_{\ell + 1}} : J),$ as desired.

Next, we prove (2). Let $x \in (I : J)^{\boxed{\#}}$. We may write $x \in (I : J)^{\#_{\ell}}$ for some $\ell \geq 1$. Then $x \in  (I^{\#_\ell} : J) \subseteq  (I^{\boxed{\#}} : J)$, as desired.

To prove (3), let $x \in (I : J)^\natural$. We may write $x^n \in ((I : J)^n)^{\boxed{\#}}$ for some $n \geq 1$. But $(I : J)^n \subseteq (I^n : J^n)$, so $x^n \in (I^n : J^n)^\boxed{\#} \subseteq ((I^n)^\boxed{\#} : J^n)$, by (2). Hence, for $j \in J$, we have $(xj)^n = x^nj^n \in (I^n)^\boxed{\#}$, which implies $xj \in I^\natural$. Thus $x \in (I^\natural : J)$, as desired.
\end{proof}

\section{Root closure of the Rees algebra}\label{sec-rees}

 Let $A$ be a ring, let $I \subseteq A$ be an ideal, and let $t$ be an indeterminate over $A$. In this section, we will consider the extension of rings $A[It] \subseteq A[t]$, where $A[It]$ has the following $\mathbb{N}$-graded structure: $$ A[It] = \bigoplus\limits_{n \in \mathbb{N}} I^nt^n = A \oplus It \oplus I^2t^2 \oplus \dots$$

We investigate the root closure $A[It]^{\boxed{\#}}$ of $A[It]$ in $A[t]$. We will prove the following:

\begin{thm}\label{thm}
Let $A$ be a ring, let $I \subseteq A$ be an ideal, and let $t$ be an indeterminate over $A$. Consider the extension of rings $A[It] \subseteq A[It]^{\boxed{\#}} \subseteq A[t]$. Define $$S = \bigoplus\limits_{n \in \mathbb{N}} (I^n)^\natural t^n = A \oplus I^\natural t \oplus (I^2)^\natural t^2 \oplus \dots .$$ The subring $S$ of $A[t]$ is equal to the root closure $A[It]^{\boxed{\#}}$.
\end{thm}

Before proving Theorem \ref{thm}, we will need several lemmas. Lemma \ref{lem-thm-containment} will be utilized to prove the contaiment $S \subseteq A[It]^{\boxed{\#}}$. To prove $A[It]^{\boxed{\#}} \subseteq S$, we will show $S$ is a root closed subring of $A[t]$ containing $A[It]$. We will use a result from Roitman \cite[Corollary 1.5]{Ro} which allows us to check the root-closedness of $S$ in $A[t]$ by focusing on the \textit{homogeneous} elements $x$ of $A[t]$ such that $x^n \in S$, instead of any general such element.

\begin{lem}\label{lem-thm-containment}
Let $A$ be a ring, let $I \subseteq A$ be an ideal, and let $t$ be an indeterminate over $A$. Then for all $\ell \geq 1, n \in \mathbb{N}$, the following containment holds: $$ (I^n)^{\#_\ell}t^n \subseteq A[It]^{\#_\ell}.$$
\end{lem}
\begin{proof}
We proceed by induction on $\ell$. 

We first consider the case where $\ell = 1$. Fix $n \in \mathbb{N}$. We need to show $(I^n)^\#t^n \subseteq A[It]^\#$. Let $x \in A$ be such that $x^m \in (I^n)^m$ for some $m \geq 1$. The ideal $(I^n)^\#$ of $A$ is generated by elements of this form, so it suffices to show $xt^n \in A[It]^{\#}$. But $(xt^n)^m = x^mt^{nm} \in I^{nm}t^{nm} \subseteq A[It]$, so $xt^n \in A[It]^\#$, as desired.

Now, fix $\ell \geq 1$ and assume $(I^n)^{\#_\ell}t^n \subseteq A[It]^{\#_\ell}$ for all $n \in \mathbb{N}$. Fix $n \in \mathbb{N}$. We need to show $(I^n)^{\#_{\ell+1}}t^n \subseteq A[It]^{\#_{\ell+1}}$. Let $x \in A$ be such that $x^m \in ((I^n)^{\#_\ell})^m$ for some $m \geq 1$. The ideal $(I^n)^{\#_{\ell + 1}}$ of $A$ is generated by elements of this form, so it suffices to show $xt^n \in A[It]^{\#_{\ell + 1}}$.  But $ ((I^n)^{\#_{\ell}})^m \subseteq  (I^{nm})^{\#_{\ell}}$ by Lemma \ref{lem-contains}, so  $x^m \in (I^{nm})^{\#_{\ell}}$. This implies $(xt^n)^m = x^mt^{nm} \in (I^{nm})^{\#_{\ell}}t^{nm} \subseteq A[It]^{\#_\ell}$. Hence $xt^n \in  A[It]^{\#_{\ell+1}},$ as desired.
\end{proof}

\begin{lem}\cite[Corollary 1.5]{Ro}\label{lem-thm-rev-containment}
Let $A \subseteq B$ be an extension of $\mathbb{N}$-graded rings. Then $A^{\boxed{\#}}$ is the smallest subring $R$ of $B$ containing $A$ such that if $x^n \in R$ for some homogeneous $x \in B$ and $n \geq 1$, then $x \in R$.
\end{lem}

We now prove Theorem \ref{thm}.
\begin{proof}
We begin by proving $S \subseteq A[It]^{\boxed{\#}}$. Let $at^n \in S$ with $n \in \mathbb{N}$ and $a \in (I^n)^\natural$. It suffices to show $at^n \in A[It]^{\boxed{\#}}$. We may write $a^m \in ((I^n)^m)^{\boxed{\#}}$ for some $m \geq 1$. From this, we may write $a^m \in ((I^n)^m)^{\#_\ell} = (I^{nm})^{\#_\ell}$ for some $\ell \geq 1$. Now $(at^n)^m = a^mt^{nm} \in (I^{nm})^{\#_\ell}t^{nm}$. By Lemma \ref{lem-thm-containment}, we see $(at^n)^m \in A[It]^{\#_{\ell}}$, so $at^n \in A[It]^{\#_{\ell+1}} \subseteq A[It]^{\boxed{\#}}$, as desired.

Next we prove $A[It]^{\boxed{\#}} \subseteq S$. Since $A[It]^{\boxed{\#}}$ is the smallest root closed subring of $A[t]$ containing $A[It]$, it suffices to show $S$ is root closed in $A[t]$. To this end, first note that $S \subseteq A[t]$ is an extension of $\mathbb{N}$-graded rings.  Suppose $x^n \in S$ for some homogeneous $x \in A[t]$ and $n \geq 1$.  If we can prove that $x \in S$, then we will have shown by Lemma \ref{lem-thm-rev-containment} that $S = S^{\boxed{\#}}$;  i.e. that $S$ is root closed in $A[t]$. To this end, write $x = at^d$ for some $a \in A$ and $d \in \mathbb{N}$. Since $a^nt^{dn} = (at^d)^n  \in S$, we see $a^n \in (I^{dn})^\natural$. We may write $(a^{n})^{\ell} \in ((I^{dn})^\ell)^\boxed{\#}$ for some $\ell \geq 1$. Then $a^{n\ell} \in ((I^{d})^{n\ell})^\boxed{\#}$, which shows $a \in (I^d)^\natural$. Thus $at^d \in S$, as desired.
\end{proof}

\section{Examples}\label{sec-ex}

Let $R$ be a ring and let $I \subseteq R$ be an ideal. In Remark \ref{rem-containment} we describe the following chain of containments which always hold: $$I \underset{\text{Ex. } \ref{exa}}\subseteq I^{\#} \underset{\text{Ex. } \ref{exb}}\subseteq I^\boxed{\#} \underset{\text{Ex. } \ref{exc}}\subseteq I^\natural \underset{\text{Ex. } \ref{exd}}\subseteq \overline{I}.$$ For each containment, we have listed an example where we provide a ring $R$ and ideal $I \subseteq R$ such that the containment is proper.

\begin{Example}\label{exa}
Consider the polynomial ring $R = \mathbb{C}[x,y]$ and ideal $I = \langle x^2, y^2 \rangle \subseteq R$. As described in Remark \ref{rem-monomial}, since $I$ is a monomial ideal we have that $I^{\#} = \overline{I}$. Notice that $xy \in \overline{I}$, since $(xy)^2 - x^2y^2 = 0$ is an equation of integral dependence for $xy$ over $I$.  However, $xy \not \in I$, so we see that $I \neq I^\#$.
\end{Example}

Before proceeding, we prove a lemma which will be utilized heavily in the following examples.

\begin{lem}\label{lem-ex}
Let $k$ be a field, and let $R$ be a quotient of a standard-graded $k$-algebra by a homogeneous ideal. Let $I \subseteq R$ be a homogeneous proper ideal, and let $x \in I^{\#}$ be a nonzero homogeneous element of degree $1$. Then we may write $x = a_1 + \dots + a_s$, where $a_i^n \in I^n$ for some $n \geq 1$ and $a_1, \dots, a_s$ are homogeneous of degree $1$.
\end{lem}

\begin{proof}For a nonzero element $y \in R$, let $y^*$ denote the smallest nonzero homogeneous component of $y$. Since $I$ is homogeneous, so is $I^n$. Then if $y^n \in I^n$ for some $n \geq 1$, then $(y^*)^n \in I^n$. 

Since $x \in I^{\#}$, we may write $x = a_1 + \dots + a_t$ for some $a_i \in R$ with $a_i^{n_i} \in I^{n_i}$ for some $n_i \geq 1$ for all $i$. Let $n$ denote the product of the elements $n_1, \dots, n_t$. Observe that $a_i^{n} \in I^{n}$ for all $i$.  Without loss of generality, we may assume $a_i  \neq 0$ for all $i$. Note that $\text{deg}_R(a_i^*) > 0$ for all $i$. Indeed, if $\text{deg}_R(a_i^*) = 0$ for some $i$, then $a_i^* \in k$. But $a_i^*$ is nonzero, so $a_i^*$ is a unit. Also note that $a_i^n \in I^n$ implies $(a_i^*)^n \in I^n$, so $I^n$ contains a unit. But $I^n \subseteq I$, which contradicts the properness of $I$.  Now, for some $i$, $\text{deg}_R(a_i^*) = 1$. Indeed, if we instead assume $\text{deg}_R(a_i^*) > 1$ for all $i$, taking the homogeneous component of degree $1$ of the equation $x = a_1 + \dots + a_t$ gives $x = 0$, a contradiction. 

Without loss of generality, we may reorder the elements $a_i$ so that $\text{deg}_R(a_i^*) = 1$ for $1 \leq i \leq s$ (for some some $s \leq t$) and $\text{deg}_R(a_i^*)  > 1$ otherwise. Taking the homogeneous component of degree $1$ of the equation $x = a_1 + \dots + a_t$ gives $x = a_1^* + \dots + a_s^*$. Observe that $(a_i^*)^n \in I^n$ and $\text{deg}_R(a_i^*) = 1$ for all $i$, $1 \leq i \leq s$, which completes the proof. 
\end{proof}

\begin{Example}\label{exb}
Let $R = \dfrac{\mathbb{Z}_2[A,B,T]}{\langle A^3 - B^3, T^3 - A^2B - AB^2 \rangle}$, where $A, B,$ and $T$ are indeterminates given degree $1$. Let $a, b,$ and $t$ refer respectively to the residue classes of $A, B,$ and $T$ in $R$. Let $I = \langle a \rangle \subseteq R$. We will show $I^{\#} \neq (I^{\#})^{\#}$ by proving $t \in (I^{\#})^{\#} \setminus I^\#$. Since $(I^{\#})^{\#} \subseteq I^{\boxed{\#}}$, this example also proves $I^{\#} \neq I^{\boxed{\#}}$.
\end{Example}
\begin{proof}
First note that $b^3 = a^3 \in I^3$, so $b \in I^{\#}$. Also, $t^3 = a^2b +ab^2 \in (I^{\#})^3$, so $t \in (I^{\#})^{\#}$. We will show that $t \not\in I^{\#}$. For the sake of contradiction, assume that $t \in I^{\#}$. Since $t \neq 0$ and $I$ is a proper ideal, by Lemma \ref{lem-ex} we may write $t  = x_1 + \dots + x_k$, where $x_i^n \in I^n$ for some $n \geq 1$ and $x_i$ is homogeneous of degree $1$  for all $i$, $1 \leq i \leq k$.

Note that $t \not\in \langle a, b \rangle$. Indeed, supposing otherwise, we find that $$T -  u(A,B,T)A - v(A,B,T)B \in \langle A^3 - B^3, T^3 - A^2B - BA^2 \rangle$$ for some $u, v \in \mathbb{Z}_2[A, B, T]$. Since the non-zero elements of the displayed ideal have no homogeneous terms of degree smaller than $3$, taking the homogeneous component of degree $1$ of the displayed containment provides $T - u_0A - v_0B = 0$, where $u_0, v_0 \in \mathbb{Z}_2$ are the constant terms of $u$ and $v$ respectively. This is a contradiction; no such relation between $T, A,$ and $B$ holds in $\mathbb{Z}_2[A,B,T]$. 

It follows that at least one element $x_i$, $1 \leq i \leq k$, is of the form $x_i =  t + \ell_1 a + \ell_2 b$ for some $\ell_1, \ell_2 \in \mathbb{Z}_2$. Indeed, otherwise the equation $t = x_1 + \dots + x_k$ shows $t \in \langle a,b \rangle$. However, we claim no element of the form $t + \ell_1 a + \ell_2 b$  satisfies $(t + \ell_1 a + \ell_2 b)^n \in I^n = \langle a^n \rangle$ for any $n \geq 1$, $\ell_1, \ell_2 \in \mathbb{Z}_2$, a contradiction. By taking the homogeneous component of degree $n$ of the last containment, we find that $(t + \ell_1 a + \ell_2 b)^n = ca^n$ for some $c \in \mathbb{Z}_2$. We check several cases:

\underline{Case 1:} $(c = 0)$. Since $c = 0$, we find that $(t + \ell_1a + \ell_2b)^n = 0$, so $t + \ell_1a + \ell_2 b$ belongs to the nilradical of $R$. Lifting to $\mathbb{Z}_2[A,B,T]$, we find $T + \ell_1A + \ell_2B \in \sqrt{\langle A^3 - B^3, T^3 - A^2B - AB^2 \rangle}$. It can be checked computationally that $$\sqrt{\langle A^3 - B^3, T^3 - A^2B - AB^2 \rangle} = \langle A^3 - B^3, T^3 - A^2B - AB^2, A^2T + ABT + B^2T \rangle.$$ Since this ideal is generated by homogeneous elements of degree $3$, the fact that it contains $T + \ell_1A + \ell_2 B$ implies $T + \ell_1A + \ell_2B = 0$. As established earlier, this is not possible.

\underline{Case 2:} $(c = 1,$ $\ell_1 = \ell_2 = 0)$. Since $t^n = a^n$ for some $n \geq 1$, by lifting to $\mathbb{Z}_2[A,B,T]$ we find that $T^n - A^n = u(A,B,T)(A^3 - B^3) + v(A,B,T)(T^3 - A^2B - AB^2)$ for some $u, v \in \mathbb{Z}_2[A,B,T]$. By sending $B \rightarrow A$, we find that $T^n - A^n = v(A,A,T)T^3$, which is not possible.

\underline{Case 3:} $(c = 1,$ $\ell_1 = \ell_2 = 1)$. Since $(t + a + b)^n = a^n$ for some $n \geq 1$, by lifting to $\mathbb{Z}_2[A,B,T]$ we find that $(T + A + B)^n - A^n = u(A,B,T)(A^3 - B^3) + v(A,B,T)(T^3 - A^2B - AB^2)$ for some $u, v \in \mathbb{Z}_2[A,B,T]$. By sending $B \rightarrow A$, we find that $T^n - A^n = v(A,A,T)T^3$, which is not possible.

\underline{Case 4:} $(c = 1,$ $\ell_1 =0,$ $\ell_2 = 1)$. Since $(t + b)^n = a^n$ for some $n \geq 1$, by lifting to $\mathbb{Z}_2[A,B,T]$ we find that $(T + B)^n - A^n = u(A,B,T)(A^3 - B^3) + v(A,B,T)(T^3 - A^2B - AB^2)$ for some $u, v \in \mathbb{Z}_2[A,B,T]$. Let $\omega$ denote a cubic root of unity distinct from $1$. By sending $B \rightarrow \omega A$ and $T \rightarrow \omega A$, we find that $$A^n = u(A, \omega A, \omega A)\underbrace{(A^3 - \omega^3 A^3)}_{=A^3 - A^3 = 0} + v(A, \omega A, \omega A)(\omega^3A^3 - \omega A^3 - \omega^2 A^3).$$

Observe also that $(\omega^3A^3 - \omega A^3 - \omega^2 A^3) = (1 + \omega + \omega^2)A^3 = 0A^3 = 0$, since $\omega \neq 1$ satisfies $ (\omega + 1)(1 + \omega +\omega^2) =  \omega^3 - 1 = 0$. Hence, by the displayed equation, we find $A^n = 0$, a contradiction.

\underline{Case 5:} $(c = 1,$ $\ell_1 =1,$ $\ell_2 = 0)$. Since $(t + a)^n = a^n$ for some $n \geq 1$, by lifting to $\mathbb{Z}_2[A,B,T]$ we find that $(T + A)^n - A^n = u(A,B,T)(A^3 - B^3) + v(A,B,T)(T^3 - A^2B - AB^2)$ for some $u, v \in \mathbb{Z}_2[A,B,T]$. Again, let $\omega$ denote a cubic root of unity distinct from $1$. By sending $B \rightarrow \omega A$ and $T \rightarrow A$, we find that $$A^n = u(A, \omega A,  A)\underbrace{(A^3 - \omega^3 A^3)}_{=A^3 - A^3 = 0} + v(A, \omega A, A)(A^3 - \omega A^3 - \omega^2 A^3).$$

Observe also that $(A^3 - \omega A^3 - \omega^2 A^3) = (1 + \omega + \omega^2)A^3 = 0A^3 = 0$, since $\omega \neq 1$ satisfies $ (\omega + 1)(1 + \omega +\omega^2) =  \omega^3 - 1 = 0$. Hence, by the displayed equation, we find $A^n = 0$, a contradiction.
\end{proof}

\begin{Example}\label{exc}
Let $R = \dfrac{\mathbb{Z}_2[A,B,C,T]}{\langle A^3 - C^6, B^3 - C^6, T^2 - A - B \rangle}$, where $A$ and $B$ are indeterminates given degree $2$, and $C$ and $T$ are indeterminates given degree $1$. Let $a, b, c,$ and $t$ refer respectively to the residue classes of $A, B, C,$ and $T$ in $R$. Let $I = \langle c \rangle \subseteq R$. We will show $I^{\boxed{\#}} \neq I^\natural$ by proving $t \in I^\natural \setminus I^{\boxed{\#}}$.
\end{Example}
\begin{proof}
First note that $a^3 = c^6 \in (I^2)^3$, so $a \in (I^2)^{\#}$. Also, $b^3 = c^6 \in (I^2)^3$, so $b \in (I^2)^{\#}$. Now $t^2 = a + b \in (I^2)^{\#} \subseteq (I^2)^{\boxed{\#}}$, so $t \in I^\natural$. 

Next, we prove $t \not\in I^{\#}$. Suppose for the sake of contradiction that $t \in I^{\#}$. Since $t \neq 0$ and $I$ is a proper ideal, by Lemma \ref{lem-ex} we may write $t  = x_1 + \dots + x_k$, where $x_i^n \in I^n$ for some $n \geq 1$ and $x_i$ is homogeneous of degree $1$  for all $i$, $1 \leq i \leq k$.

Note that $t \not\in \langle c \rangle$. Indeed, supposing otherwise, we find that $$T -  u(A,B,C,T)C \in \langle A^3 - C^6, B^3 - C^6, T^2 - A - B \rangle$$ for some $u \in \mathbb{Z}_2[A,B,C,T]$. Since the non-zero elements of the displayed ideal have no homogeneous terms of degree smaller than $2$, taking the homogeneous component of degree $1$ of the displayed containment provides $T - u_0C = 0$, where $c_0 \in \mathbb{Z}_2$ is the constant term of $u$. This is a contradiction; no such relation between $T$ and $C$ holds in $\mathbb{Z}_2[A,B,C,T]$.

It follows that at least one element $x_i$, $1 \leq i \leq k$, is of the form $x_i =  t + \ell c$ for some $\ell \in \mathbb{Z}_2$. Indeed, otherwise the equation $t = x_1 + \dots + x_k$ shows $t \in \langle c \rangle$. However, we claim no element of the form $t + \ell c$  satisfies $(t + \ell c)^n \in I^n = \langle c^n \rangle$ for any $n \geq 1$, $\ell  \in \mathbb{Z}_2$, a contradiction. By taking the homogeneous component of degree $n$ of the last containment, we find that $(t + \ell c)^n = rc^n$ for some $r \in \mathbb{Z}_2$. We check several cases:

\underline{Case 1:} ($r = 0$). Since $r = 0$, we find $t + \ell c$ belongs to the nilradical of $R$. Lifting to $\mathbb{Z}_2[A,B,C,T]$, we find $T + \ell  C \in \sqrt{\langle A^3 - C^6, B^3 - C^6, T^2 - A - B \rangle}$. It can be checked computationally that $$\sqrt{\langle A^3 - C^6, B^3 - C^6, T^2 - A - B \rangle} = \langle T^2 - A - B, T^4 + TC^3, B^2T + BT^3 + T^2C^3, B^3 + C^6 \rangle.$$ Since this ideal is generated by homogeneous elements of degree larger than 1, the fact that it contains $T + \ell C$ implies $T + \ell C = 0$. As established earlier, this is not possible.

\underline{Case 2:} ($r = 1$, $\ell = 0$). Since $t^n = c^n$ for some $n \geq 1$, by lifting to $\mathbb{Z}_2[A,B,C,T]$, we find that $$T^n - C^n = \alpha_1(A,B,C,T)(A^3 - C^6) + \alpha_2(A,B,C,T)(B^3 - C^6) + \alpha_3(A,B,C,T)(T^2 - A - B)$$ for some $\alpha_1, \alpha_2, \alpha_3 \in \mathbb{Z}_2[A,B,C,T]$. Sending $A \rightarrow C^2$ and $B \rightarrow C^2$ gives $$T^n - C^n = \alpha_1\underbrace{(C^6 - C^6)}_{=0} + \alpha_2\underbrace{(C^6 - C^6)}_{ = 0} + \alpha_3(T^2 - C^2 - C^2) = \alpha_3T^2.$$ Hence $T^n - C^n \in \langle T^2 \rangle$, which is not possible.

\underline{Case 3:} ($r = 1$, $\ell = 1$). Since $(t + c)^n = c^n$ for some $n \geq 1$, by lifting to $\mathbb{Z}_2[A,B,C,T]$, we find that $$(T + C)^n - C^n = \alpha_1(A,B,C,T)(A^3 - C^6) + \alpha_2(A,B,C,T)(B^3 - C^6) + \alpha_3(A,B,C,T)(T^2 - A - B)$$ for some $\alpha_1, \alpha_2, \alpha_3 \in \mathbb{Z}_2[A,B,C,T]$. By replacing $\mathbb{Z}_2$ with its algebraic closure, we may assume the underlying field contains two distinct cubic roots of unity $\omega_1 \neq 1$ and $\omega_2 \neq 1$. Sending $A \rightarrow \omega_1C^2$, $B \rightarrow \omega_2C^2$, and $T \rightarrow C$ gives $$C^n = \alpha_1\underbrace{(\omega_1^3C^6 - C^6)}_{=C^6 - C^6=0} + \alpha_2\underbrace{(\omega_2^3C^6 - C^6)}_{=C^6 - C^6 = 0} + \alpha_3(C^2 - \omega_1C^2 - \omega_2C^2).$$

Note that $C^2 - \omega_1C^2 - \omega_2C^2 = C^2(1 + \omega_1 + \omega_2) = 0$. The right side of the displayed equation becomes $0$, a contradiction.

We have shown $t \not \in I^{\#}$. To show $t \not\in I^{\boxed{\#}}$, it suffices to show $(I^{\#})^{\#} = I^{\#}$. To this end, we first compute $I^{\#}$. We claim $\langle a, b, c \rangle = I^{\#}$. For the forward containment, we will prove $a, b, c \in I^{\#}$. First, note $c \in I \subseteq I^{\#}$. Next, $a^3 = c^6 \in I^6 \subseteq I^3$, so $a \in I^{\#}$. Similarly, $b^3 = c^6 \in I^6 \subseteq I^3$, so $b \in I^{\#}$. For the reverse containment, we work by contradiction. Assume there exists $r \in R$ such that $r \in I^{\#} \setminus \langle a, b, c \rangle$. By removing ``terms" of $r$ contained $\langle a, b, c \rangle$, we may reduce to the case where $r$ is of the form $c_0 + c_1t + \dots + c_mt^m$ for some $c_i \in \mathbb{Z}_2$. In fact, for $i \geq 2$, $t^i = t^{2}t^{i-2} = (a +b)t^{i-2} \in \langle a, b, c \rangle$. Hence we may further reduce to the case where $r = c_0 + c_1t$. It can not be the case that $c_0 = 1$. Indeed, otherwise $r = c_0 + c_1t \in I^{\#} \subseteq \sqrt{I}$ implies $c_0 = 1 \in \sqrt{I}$, since $\sqrt{I}$ is a homogeneous ideal. But $I$ is proper, so $\sqrt{I}$ is too, contradiction. Hence $r = c_1t$. If $c_1 = 0$, $0 = r \in \langle a, b, c \rangle$, a contradiction. Hence $r = t$. But we just proved $r  = t \not \in I^{\#}$, a contradiction. 

To prove $(I^{\#})^{\#} = I^{\#}$, we must prove $\langle a , b, c \rangle^{\#} = \langle a , b, c \rangle$. We work by contradiction. let $x \in R$ be such that $x \in \langle a, b ,c \rangle^{\#} \setminus \langle a, b, c \rangle$ . Working as above, by removing ``terms" of $x$ contained $\langle a, b, c \rangle$, we may reduce to the case where $x$ is of the form $c_0 + c_1t$ for some $c_i \in \mathbb{Z}_2$. It can not be the case that $c_0 = 1$. Indeed, otherwise $x = c_0 + c_1t \in \langle a, b, c \rangle ^{\#} \subseteq \sqrt{\langle a, b, c \rangle}$  implies $c_0 = 1 \in \sqrt{\langle a, b, c \rangle}$, a contradiction. Hence $x = c_1t$. If $c_1 = 0$, $0 = x \in \langle a, b, c \rangle$, a contradiction. Hence $x = t \in \langle a, b, c \rangle ^{\#}$. Since $t \neq 0$ and $\langle a, b, c \rangle$ is a proper ideal, we may apply Lemma \ref{lem-ex} to write $t = y_1 + \dots + y_s$, where $m \geq 1$, $y_i^m \in \langle a, b, c \rangle^m$ for all $i$, $1 \leq i \leq s$, and $y_i$ is homogeneous of degree $1$. At least one element $y_i$, $1 \leq i \leq s$, is of the form $y_i =  t + \ell 'c$ for some $\ell ' \in \mathbb{Z}_2$. Indeed, otherwise the equation $t = y_1 + \dots + y_s$ shows $t \in \langle c \rangle$, a contradiction. 

Fix $i$ such that $y_i =  t + \ell 'c$ for some $\ell ' \in \mathbb{Z}_2$. Since $\text{deg}_R(a) = \text{deg}_R(b) = 2$ and $\text{deg}_R(c) = 1$, the degree $m$ homogeneous part of $\langle a, b, c \rangle^m$ can be seen to be $\{ 0, c^m \}$. Now taking the degree $m$ homogeneous component of the equality $y_i^m = (t + \ell ' c)^m \in \langle a, b, c \rangle ^m$, we see that $(t + \ell ' c)^m \in \langle c^m \rangle$. Earlier, in a different context, we proved this was impossible. 
\end{proof}

\begin{Example}\label{exd}
Let $R = \dfrac{\mathbb{Z}_2[A,T]}{\langle A^2 - AT \rangle }$, where $A$ and $T$ are indeterminates given degree $1$. Let $a$ and $t$ refer respectively to the residue classes of $A$ and $T$ in $R$. Let $I = \langle t \rangle$. We will show $I^\natural \neq \overline{I}$ by proving $a \in \overline{I} \setminus I^\natural$.
\end{Example}
\begin{proof}
First, observe that $a$ satisfies the equation of integral dependence $X^2 + tX = 0$ over $I$, so $a \in \overline{I}$. To prove $a \not\in I^\natural$, we first prove $a \not\in I$. Next, we prove $I^m = (I^m)^{\#}$ for all $m \geq 1$. From this, it follows $I^m = (I^m)^{\boxed{\#}}$ for all  $m \geq 1$. Then $a \in I^\natural$ implies $a^m \in (I^m)^{\boxed{\#}} = I^m$ for some $m \geq 1$, so $a \in I^{\#} = I$, a contradiction; hence $a \not\in I^\natural$.

For the sake of contradiction, suppose $a \in I = \langle t \rangle$. Taking the homogeneous component of degree $1$ of this containment, we find $a = ct$ for some $c \in \mathbb{Z}_2$. Lifting this equation to $\mathbb{Z}_2[A,T]$, we find $A - cT \in \langle A^2 - AT \rangle $. The latter ideal contains no nonzero homogeneous elements of degree $1$, so $A - cT = 0$, which is impossible. Hence $a \not \in \langle I \rangle$.

 Fix $m \geq 1$. It is clear that $I^m \subseteq (I^m)^{\#}$. For the reverse containment, let $x \in (I^m)^{\#}$. If $x = 0$, then $x \in I^m$ and we're done. So, assume $x \neq 0$. Set $\ell = \text{deg}_R(x^*)$. We proceed  by cases:

\underline{Case 1:} $(\ell > m ).$ Since $a^2 = at$, we have $\langle a, t \rangle^m \subseteq \langle at^{m-1}, t^m \rangle \subseteq \langle t^{m - 1} \rangle$.  In this case, each nonzero homogeneous component $y$ of $x$ has $\text{deg}_R(y) > m$. But $y \subseteq \langle a, t \rangle^{\text{deg}_R y} \subseteq \langle t^m \rangle$, which then implies $x \in \langle t^m \rangle$.

\underline{Case 2:} $(\ell < m ).$ Since $x \in (I^m)^{\#}$, we may write $x = x_1 + \dots + x_r$, where $x_i^{n_i} \in \langle t^m \rangle^{n_i}$ for some $n_i \geq 1$ for all $i$. Let $n$ denote the product of the elements $n_1, \dots, n_r$ and observe that $x_i^{n} \in \langle t^m \rangle^{n}$ for all $i$.  We may assume without loss of generality that $x_i \neq 0$ for all $i$. First note that $\text{deg}_R(x_i^*) \geq m$ for all $i$.  Suppose otherwise that $\text{deg}_R(x_i^*) < m$ for some $i$. Since $\langle t^{mn} \rangle$ is a homogeneous ideal, $x_i^n \in \langle t^m \rangle^n = \langle t^{mn} \rangle$ implies $(x_i^*)^n \in \langle t^{mn} \rangle$. The ideal $\langle t^{mn} \rangle$ contains no nonzero homogeneous elements of degree less than $mn$, and $\text{deg}_R\big((x_i^*)^n\big) \leq n\cdot \text{deg}_R(x_i^*)< mn$. Thus, taking the homogeneous component of degree $\text{deg}_R\big((x_i^*)^n\big)$ of the inclusion $(x_i^*)^n \in \langle t^{mn} \rangle$, we find $(x_i^*)^n = 0$. It can be checked computationally that $\sqrt{ \langle A^2 - AT \rangle} = \langle A^2 - AT \rangle$. Therefore $R$ is reduced, so $x_i^* = 0$, a contradiction. Hence $\text{deg}_R(x_i^*) \geq m$ for all $i$. Now, taking the degree $\ell$ homogeneous component of the equality $x = x_1 + \dots + x_r$, we find $x^* = 0$, a contradiction.

\underline{Case 3:} $(\ell = m ).$ Since $x \in (I^m)^{\#}$, we may write $x = x_1 + \dots + x_r$, where $x_i^n \in \langle t^m \rangle^n$ for some $n \geq 1$, for all $i$. We may assume without loss of generality that $x_i \neq 0$ for all $i$. Exactly as in the $\ell < m$ case, we find $\text{deg}_R(x_i^*) \geq m$ for all $i$. For at least one $i$, the element $x_i$ has $\text{deg}_R(x_i^*) = m$. Otherwise, if $\text{deg}_R(x_i^*) > m$ for all $i$, then taking the homogeneous component of degree $m$ of the equation $x = x_1 + \dots + x_r$ provides $x = 0$, a contradiction. Without loss of generality, we may reorder the elements $x_i$ so that $\text{deg}_R(x_i^*) = m$ for $1 \leq i \leq s$ (for some some $s \leq r$) and $\text{deg}_R(x_i^*)  > m$ otherwise. Taking the homogeneous component of degree $m$ of the equation $x = x_1 + \dots + x_r$ gives $x = x_1^* + \dots + x_s^*$. After a renaming, we may write $x = x_1 + \dots + x_s$ for some $x_i \in R$ such that $x_i^n \in \langle t^m \rangle ^n$ and $x_i$ is homogeneous of degree $m$.

To show $x \in I^m$ as desired, it suffices to show $x_i \in I^m = \langle t^m \rangle$ for all $i$. Without loss of generality, we show $x_1 \in \langle t^m \rangle$. Since $x_1$ is homogeneous of degree $m$, then $$x_1 \in \langle a, t \rangle^m \subseteq \langle t^{m-1} \rangle.$$ Taking the homogeneous component of degree $m$ of this inclusion gives $x_1 = \alpha t^{m-1}$ for some $\alpha \in R$ homogeneous of degree $1$. Now $\alpha^nt^{nm - n} = x_1^n \in \langle t^m\rangle^n = \langle t^{mn} \rangle.$ Note that $t$ is a non-zero-divisor in $R$. To see this, it can be checked computationally that the minimal prime ideals of $R$ are $\langle t - a\rangle$ and $\langle a \rangle$. Lifting to $\mathbb{Z}_2[A,T]$, we find $$T \not\in \langle T - A \rangle + \langle A^2 - AT \rangle = \langle T - A \rangle \text{ and } T \not \in \langle A \rangle + \langle A^2 - AT \rangle  = \langle A \rangle.$$ Hence $t \not \in \langle t - a \rangle$ and $t \not \in  \langle a \rangle$, so $t \not \in \langle t -a \rangle \cup \langle a \rangle $, as desired. We may cancel $t^{nm - n}$ in the inclusion $\alpha^n t^{nm - n} \in \langle t^{mn}\rangle$ to find $\alpha^n \in \langle t^n \rangle$. We claim $\alpha \in \langle t \rangle$.  Write $\alpha = c_0a + ct$ for some $c_0, c \in \mathbb{Z}_2$. If $c_0 = 0$, then we're done. So, assume $c_0 = 1$ and note that we have $\alpha = a + ct $ for some $c \in \mathbb{Z}_2$.  Taking the homogeneous component of degree $n$ of the inclusion $\alpha^n \in \langle t^n \rangle$, we find $(a + ct)^n = dt^n$ for some $d \in \mathbb{Z}_2$. We prove this is impossible by checking cases:

\underline{Case 1:} $(d = 0)$. Since $R$ is reduced, we find $a + ct = 0$. Lifting to $\mathbb{Z}_2[A,T]$, we recover $A + cT \in \langle A^2 - AT \rangle$. This ideal contains no nonzero homogeneous elements of degree $1$, so $A + cT = 0$. This is not possible.

\underline{Case 2:} ($c = 0$, $d = 1$).  Here, $a^n = t^n$. Lifting, we find $A^n - T^n \in \langle A^2 - AT \rangle$. Sending $A \rightarrow 0$, we find $-T^n = 0$, a contradiction.

\underline{Case 3:} ($c = 1$, $d = 1$). Here, $(a + t)^n = t^n$. We find $(A + T)^n - T^n \in \langle A^2 - AT \rangle$. Sending $A \rightarrow T$, we find $-T^n \in  \langle T^2 - T^2 \rangle = \langle 0 \rangle$, a contradiction.
\end{proof}

\end{document}